\documentclass[10pt]{amsart}

\usepackage{stmaryrd}
\usepackage{amsmath}
\usepackage{amsthm}
\usepackage{amssymb}
\usepackage{enumerate}
\usepackage{enumitem}
\usepackage{mathrsfs}
\usepackage{verbatim}
\usepackage[all]{xy}
\usepackage{color}
\usepackage{graphicx}

\usepackage{pict2e}
\usepackage{fancybox}
\usepackage{tikz}

 \usepackage[usenames,dvipsnames]{pstricks}
 \usepackage{epsfig}

\theoremstyle{plain}
\newtheorem{lemma}[subsection]{Lemma}
\newtheorem{proposition}[subsection]{Proposition}
\newtheorem{theorem}[subsection]{Theorem}
\newtheorem{corollary}[subsection]{Corollary}

\newtheorem*{theorem*}{Theorem}
\newtheorem*{proposition*}{Proposition}

\theoremstyle{definition}

\newtheorem{example}[subsection]{Example}
\newtheorem{remark}[subsection]{Remark}

\numberwithin{equation}{section}

\DeclareMathOperator*{\colim}{colim}
\DeclareMathOperator*{\hocolim}{hocolim}
\DeclareMathOperator*{\holim}{holim}

\newcommand{\Cat}{\mathit{Cat}}
\newcommand{\Top}{\mathit{Top}} 

\newcommand{\cA}{\mathcal A} 
\newcommand{\cB}{\mathcal B}

\newcommand{\cX}{\mathcal X}
\newcommand{\cY}{\mathcal Y}
\newcommand{\cZ}{\mathcal Z}

\newcommand{\scC}{\mathscr{C}} 
\newcommand{\scD}{\mathscr{D}}
\newcommand{\scG}{\mathscr{G}}
\newcommand{\scK}{\mathscr{K}}
\newcommand{\scM}{\mathscr{M}}
\newcommand{\scP}{\mathscr{P}}
\newcommand{\scQ}{\mathscr{Q}}
\newcommand{\scJ}{\mathscr{J}}

\newcommand{\ang}[1]{\langle#1\rangle}
\newcommand{\num}[1]{\lvert#1\rvert}

\newcommand{\op}{\mathrm{op}}
\newcommand{\id}{\mathrm{id}}
\newcommand{\xr}{\xrightarrow}
\newcommand{\xl}{\xleftarrow}
\newcommand{\Max}{\operatorname{Max}}
\newcommand{\peq}{\preccurlyeq}
\newcommand{\speq}{\prec}
\newcommand{\suarrow}{\shortuparrow}
\newcommand{\sdarrow}{\shortdownarrow}

\title{Combinatorial and homotopical aspects of $E_n$-operads}
\date{\today}

\author{Christian Schlichtkrull} \address{Christian Schlichtkrull,
    Department of Mathematics, University of Bergen, P.O. Box 7803, N-5020 Bergen, Norway} \email{christian.schlichtkrull@uib.no}

\begin{document}

\begin{abstract}
We show that a certain class of categorical operads give rise to $E_n$-operads after geometric realization. The main arguments are purely combinatorial and avoid the technical topological assumptions otherwise found in the literature.
 \end{abstract}

 \subjclass[2010]{18M75, 55P48}
\maketitle

\section{Introduction}
The notion of an operad was introduced by May~\cite{May72} as a useful devise for analyzing multiplicative structures on topological spaces with applications to the theory of iterated loop spaces. Of particular importance for the analysis of $n$-fold loop spaces is the  little $n$-cubes operad $\scC_n$ of Boardman and Vogt~\cite{Boardman-Vogt}. One says that a topological operad $\scD$ is an $E_n$-operad if there exists a chain of equivalences of the form
\[
\scD\xl{\simeq} \bullet\xr{\simeq}\bullet\xr{\simeq} \dots \xl{\simeq}\bullet \xr{\simeq}{\scC_n},
\] 
where each $\bullet$ represents a topological operad and each arrow an equivalence of such. Some authors also require an $E_n$-operad to be $\Sigma$-free in the sense that the symmetric groups act freely on the spaces of the operad; all the $E_n$-operads we consider will be $\Sigma$-free in this sense. The paper \cite{BM23} contains a useful survey of $E_n$-operads appearing in the literature with new proofs of many results.

There is a corresponding notion of a categorical operad and a process of geometric realization that turns a categorical operad $\scP$ into a topological operad $\num{\scP}$. Categorical operads appear in many contexts and it is often interesting to know if their geometric realizations are $E_n$-operads. A striking example of this situation occurs in the work of Balteanu-Fiedorowicz-Schw\"anzl-Vogt~\cite{BFSV} on iterated monoidal categories. Motivated by an attempt to find a categorical counterpart to the notion of an $n$-fold loop space, these authors introduce a categorical operad $\scM_n$ whose algebras can be identified with the so-called $n$-fold monoidal categories. The comparison to $n$-fold loop spaces is then enabled by the fact that
$\num{\scM_n}$ is an $E_n$-operad. 

An operad defined as in \cite{May72} (sometimes called a reduced operad) has degeneracy operators which makes it a contravariant functor on the category of non-empty finite sets and injective functions. Following Berger \cite{Berger97}, we use the term preoperad for such a contravariant functor which may in general not possess an operad product. As we recall in Section~\ref{sec:Operads-labels-orientations}, there are inclusions of categorical (pre)operads 
\begin{equation}\label{eq:intro-(pre)operad-inclusions}
\scM_n^{\suarrow}\to \scM_n\to\scK_n\to\scK_n^e,
\end{equation}
where $\scM_n$ is as above, $\scK_n$ is Berger's complete graph operad \cite{Berger97}, and $\scK_n^e$ is the extended version of the latter introduced in \cite{BFV07}. The first term $\scM_n^{\suarrow}$ is a well-known preoperad (sometimes called the Getzler-Jones preoperad) whose definition we recall in Section~\ref{subsec:restricted-preoperads} together with its dual version $\scM_n^{\sdarrow}$. By a categorical suboperad $\scP$ of $\scK_n^e$ we in general understand a family of full subcategories of the categories comprising $\scK_n^e$, such that the operad structure of the latter restricts to an operad structure on $\scP$. The notion of a categorical subpreoperad of $\scK_n^e$ is defined analogously. At the categorical level we have for each finite set $S$ a chain of inclusions of partially ordered sets
\begin{equation}\label{eq:intro-partially-ordered-sets-inclusions}
\scM_n^{\suarrow}(S)\to \scM_n(S)\to\scK_n(S)\to\scK_n^e(S)
\end{equation}
and similarly with $\scM_n^{\sdarrow}(S)$ instead of $\scM_n^{\suarrow}(S)$. The following is our main result.
\begin{theorem}\label{thm:intro-homotopy-initial-final}
Consider an inclusion $j\colon \cA\to \cB$ of full subcategories of  $\scK_n^e(S)$. If $\cA$ contains $\scM_n^{\sdarrow}(S)$, then $j$ is homotopy initial and if $\cA$ contains $\scM_n^{\suarrow}(S)$, then $j$ is homotopy final. 
\end{theorem}
We recall the notions of homotopy initial and final functors in Section~\ref{sec:homotopy-initial-final}. Since any such functor induces an equivalence after geometric realization, it follows in particular that the inclusions in \eqref{eq:intro-(pre)operad-inclusions} and \eqref{eq:intro-partially-ordered-sets-inclusions} give rise to equivalences of their geometric realizations. In Section~\ref{sec:relation-little-n-cubes} we review the argument from~\cite{BFV07} showing that $\num{\scK_n^e}$ is an $E_n$-operad. Combining this with Theorem~\ref{thm:intro-homotopy-initial-final}, we get the following.

\begin{theorem}\label{thm:intro-categorical-sub(pre)operad-En}
If $\scP$ is a categorical sub(pre)operad of $\scK_n^e$ that contains one of the preoperads  $\scM_n^{\suarrow}$ or 
$\scM_n^{\sdarrow}$, then $\num{\scP}$ is an $E_n$-(pre)operad. 
\end{theorem}
There are several ways in which these theorems improve on existing results in the literature. First of all, there are many examples of categorical sub(pre)operads of $\scK_n^e$ containing either $\scM_n^{\suarrow}$ or $\scM_n^{\sdarrow}$ and thus we get new examples of $E_n$\nobreakdash-(pre)operads. Furthermore, since our proof of Theorem~\ref{thm:intro-homotopy-initial-final} is purely combinatorial, we can dispense with the technical topological assumptions sometimes found in the literature. This is particularly relevant for the proof that $\num{\scM_n}$ is an $E_n$-operad as presented in \cite[Theorem~3.14]{BFSV}. In this proof a certain diagram is required to be Reedy cofibrant, but as we shall see in Example~\ref{ex:FMn-diagram-not-Reedy-cof}, this condition is not satisfied in general. (The same erroneous assumption is carried over in the proposed proof in \cite{BM23}.) We discuss this situation in more detail in Section~\ref{sec:relation-little-n-cubes}. Finally, the result of Theorem~\ref{thm:intro-homotopy-initial-final} is much stronger than merely asserting that these inclusions give rise to equivalences after geometric realization. As we recall in Section~\ref{sec:homotopy-initial-final}, the statement in Theorem~\ref{thm:intro-homotopy-initial-final} implies that the induced maps of homotopy limits and colimits are equivalences for all diagrams indexed by such categories. 
This is exploited in forthcoming joint work by M. Solberg and the author, where these partially ordered sets are used as index categories in connection with various forms of higher monoidal operads related to Batanin's theory of $n$-operads~\cite{Ba-Sym}. 

\subsection{Organization of the paper}
We begin in Section~\ref{sec:Operads-labels-orientations} with a detailed review of the categorical (pre)operads appearing in \eqref{eq:intro-(pre)operad-inclusions}. In Section~\ref{sec:homotopy-initial-final} we state the main results of the paper and begin the proof of Theorem~\ref{thm:intro-homotopy-initial-final} by first reducing the proof of the theorem for general $n$ to the special case $n=2$. The case $n=2$ of the theorem is then proved in Section~\ref{sec:n=2proof} by designing an algorithm for showing that a certain simplicial complex is contractible. Finally, in Section~\ref{sec:relation-little-n-cubes}, we review the relation to the little $n$-cubes operad and use this to finish the proof of Theorem~\ref{thm:intro-categorical-sub(pre)operad-En}.

\subsection*{Acknowledgements}
We would like to thank Mirjam Solberg for helpful discussions related to this work.

\subsection{Notation and conventions}\label{subsec:notation-conventions}
We write $\Cat$ for the category of small categories and $\Top$ for the category of compactly generated weak Hausdorff topological spaces. By an equivalence in $\Top$ we generally mean a weak homotopy equivalence. Many of the equivalences we consider are actual homotopy equivalences, but we shall not emphasize this. The geometric realization of a small category $\cA$ is obtained by first forming the nerve $N\cA$ and then the geometric realization $\num{N\cA}$ of this simplicial set. It will be convenient to omit $N$ from the notation and let $\num{\cA}$ denote the geometric realization. Given a small category $\cA$ and an $\cA$-diagram $F\colon \cA\to\Top$, we write $\hocolim_{\cA}F$ for the homotopy colimit of $F$ as defined in \cite{Hirschhorn}. 
We shall consider both topological operads (internal to $\Top$) and categorical operads (internal to $\Cat$) and refer to \cite{May-definitions, May72} for the relevant definitions. The geometric realization of a categorical operad $\scP$ is the topological operad $\num\scP$ obtained by geometric realization in each operad degree. A map of topological operads is said to be an equivalence if the map in each operad degree is an equivalence.

\section{Operads arising from edge labels and orientations}\label{sec:Operads-labels-orientations}
Let $S$ be a finite set. The complete graph on $S$ has as its edges all pairs $\{x,y\}$ of distinct elements from $S$. For each natural number $n\geq 1$ we shall consider a category $\scG_n(S)$ of edge labels and orientations. An object $\mu$ of $\scG_n(S)$ is given by a choice of label $\mu\{x,y\}$ in $\{1,\dots,n\}$ and orientation $\vec\mu\{x,y\}\in\{(x,y),(y,x)\}$ for each edge $\{x,y\}$. There is a morphism $\mu\to\nu$ if for each $\{x,y\}$ we have either 
\[
\begin{cases}
\vec\mu\{x,y\}=\vec\nu\{x,y\}\\
\mu\{x,y\}\leq \nu\{x,y\}
\end{cases}
\text{ or }\quad
\begin{cases}
\vec\mu\{x,y\}\neq\vec\nu\{x,y\}\\
\mu\{x,y\}< \nu\{x,y\}.
\end{cases}
\]
Thus, $\scG_n(S)$ is by definition a partially ordered set. Our convention for $S=\emptyset$ is to let $\scG_n(\emptyset)=\{0\}$, and if $S=\{x\}$ is a singleton, we write $\scG_n(\{x\})=\{x\}$, thinking of $0$ and $x$ as representing the unique ``labeling/orientation'' in each case. We usually write $\scG_n(k)$ instead of $\scG_n(\{1,\dots,k\})$.
Given finite sets $S_1, \dots, S_k$ for $k\geq 1$, there are operad structure maps
\begin{equation}\label{eq:Gn-operad-product}
\gamma\colon \scG_n(k)\times \scG_n(S_1)\times\dots\times \scG_n(S_k)\to \scG_n(S_1\sqcup\dots\sqcup S_k),
\end{equation}
where $S_1\sqcup\dots\sqcup S_k$ denotes the disjoint union of the sets. This functor takes a tuple of objects $\mu$ in 
$\scG_n(k)$ and $\nu_i$ in $\scG_n(S_i)$ for $i=1,\dots,k$, to the object with labels
\[
\gamma(\mu,\nu_1,\dots,\nu_k)\{x,y\}=
\begin{cases}
\nu_i\{x,y\},& \text{if $\{x,y\}\subseteq S_i$},\\
\mu\{i,j\},& \text{if $x\in S_i$, $y\in S_j$, $i\neq j$},
\end{cases}
\]
and orientations
\[
\vec\gamma(\mu,\nu_1,\dots,\nu_k)\{x,y\}=
\begin{cases}
\vec\nu_i\{x,y\},& \text{if $\{x,y\}\subseteq S_i$},\\
(x,y),& \text{if $x\in S_i$, $y\in S_j$, $i\neq j$, and $\vec \mu\{i,j\}=(i,j)$}.
\end{cases}
\]
Restricted to the categories $\scG_n(k)$ for $k\geq 0$ we obtain a categorical operad $\scG_n$. This has as its unit the unique object $1$ in $\scG_n(1)$ and the right action of an element $\sigma$ in the symmetric group $\Sigma_k$ is defined on labels by 
$\mu\sigma\{x,y\}=\mu\{\sigma(x),\sigma(y)\}$ and on orientations by
\[
\vec{\mu\sigma}\{x,y\}=(x,y)\quad\text{if} \quad
\vec\mu\{\sigma(x),\sigma(y)\}=(\sigma(x),\sigma(y)).
\]
The operad product of $\scG_n$ is defined as the above map $\gamma$ by giving the ordered disjoint union of ordered sets in the codomain the induced ordering. 

In the following we shall not be concerned with the operad $\scG_n$ itself, but rather with certain categorical suboperads that give rise to $E_n$-operads upon geometric realization (this is not the case for $\scG_n$).  In the following $S$ denotes a finite set as in the beginning of the section.

\subsection{The complete graph operad $\scK_n$} 
Let $\scK_n(S)$ be the full subcategory of $\scG_n(S)$ defined by the objects $\mu$ that do not contain an oriented cycle. Here an oriented cycle in $\mu$ means a sequence $x_1,\dots,x_k$ of distinct elements from $S$ for $k>1$, such that $\vec\mu\{x_i,x_{i+1}\}=(x_i,x_{i+1})$ for $1\leq i<k$ and $\vec\mu\{x_k,x_1\}=(x_k,x_1)$. Hence, if $\mu$ is an object of $\scK_n(S)$, then $\vec\mu$ defines a linear ordering of $S$. It is clear that the operad structure of $\scG_n$ restricts to define a categorical operad $\scK_n$. This operad was originally introduced by Berger \cite{Berger97}. 

\subsection{The extended complete graph operad $\scK^{e}_n$}
Let $\scK^{e}_n(S)$ be the full subcategory of $\scG_n(S)$ defined by the objects that do not contain an oriented cycle with constant label on the edges. This means that if $\mu$ is an object of $\scG_n(S)$ and $x_1,\dots,x_k$ is an oriented cycle in $\mu$, then for $\mu$ to be an object of $\scK^{e}_n(S)$ we require that the labels $\mu\{x_i,x_{i+1}\}$ for $1\leq i<k$ and $\mu\{x_k,x_1\}$ are not all equal. It is again clear that the operad structure of $\scG_n$ restricts to define a categorical operad $\scK^e_n$ which by definition contains $\scK_n$ as a suboperad. The operad $\scK^e_n$ was introduced by Brun-Fiedorowicz-Vogt in \cite{BFV07} (there denoted $\scK_n$). We shall later review some of the motivations for introducing $\scK^e_n$.

\subsection{The operad $\scM_n$ governing $n$-fold monoidal categories}
The categorical operad $\scM_n$ was introduced by Balteanu-Fiedorowicz-Schw\"anzl-Vogt in their work on $n$\nobreakdash-fold monoidal categories~\cite{BFSV}. Indeed, the defining property of $\scM_n$ is that its algebras in $\Cat$ are exactly the $n$-fold monoidal categories as defined in the mentioned paper. In this paper we shall use an equivalent definition of $\scM_n$ as a suboperad of $\scK_n$. This requires some preparations. Consider for each $1\leq i\leq n$ the object $1\Box_i2$ of $\scK_n(2)$ that to the edge $\{1,2\}$ assigns the label $i$ and the orientation $(1,2)$. Let $(S_1,S_2)$ be an ordered partition of $S$ in the sense that $S_1$ and $S_2$ are non-empty disjoint subsets of $S$ with union $S$. Restricting the operad product \eqref{eq:Gn-operad-product} to $\scK_n$ and setting the first variable equal to $1\Box_i2$, we define a product
\[
\Box_i\colon \scK_n(S_1)\times \scK_n(S_2)\to \scK_n(S)
\]
upon identifying $S_1\sqcup S_2$ with $S$ in the canonical way. Thus, given objects $\mu_1$ in $\scK_n(S_1)$ and $\mu_2$ in $\scK_n(S_2)$, the linear ordering of $S$ induced by $\mu_1\Box_i\mu_2$ is the ordered disjoint union $S_1\sqcup S_2$ of the ordered sets $S_1$ and $S_2$. Using the products $\Box_i$, the definition of $\scM_n(S)$ proceeds recursively with respect to the cardinality of $S$. To begin we again set $\scM_n(\emptyset)=\{0\}$ and $\scM_n(\{x\})=\{x\}$. For a finite set $S$ with more than one element, suppose that $\scM_n(S')$ has been defined for all proper subset $S'$ of $S$. We then define $\scM_n(S)$ to be the full subcategory of $\scK_n(S)$ in which an object is in the image of the composition
\[
\scM_n(S_1)\times \scM_n(S_2)\to \scK_n(S_1)\times \scK_n(S_2)\xr{\Box_i} \scK_n(S)
\]
for some ordered partition $(S_1,S_2)$ and some $1\leq i\leq n$. The effect of this definition is to make $\scM_n(S)$ the subcategory of ``decomposable'' objects in $\scK_n(S)$. For instance, a typical object of $\scM_3(6)$ has the form
\[
3\Box_2((2\Box_36)\Box_14)\Box_2(1\Box_15).
\]
Notice that for an object $\mu$ of $\scM_n(S)$, the order in which the elements of $S$ appear in the decomposition is the same as the linear ordering of $S$ specified by $\mu$ as an object of $\scK_n(S)$. We claim that the operad product \eqref{eq:Gn-operad-product} restricts to a product with $\scM_n$ instead of $\scG_n$. In order to verify this, consider an ordered partition $(T_1,T_2)$ of $\{1,\dots,k\}$ and suppose that we are given objects $\mu_i$ in $\scK_n(T_i)$ for $i=1,2$ and $\nu_i$ in $\scK_n(S_i)$ for $i=1,\dots,k$. Then we have the equality
\[
\gamma(\mu_1\Box_i\mu_2,\nu_1,\dots,\nu_k)
=\gamma(\mu_1,\ang{\nu_t}_{t\in T_1})\Box_i\gamma(\mu_2,\ang{\nu_t}_{t\in T_2})
\]
in which the definition of the terms on the right hand side is supposed to be self-explanatory. Using this, it follows by induction on $k$ that the operad product \eqref{eq:Gn-operad-product} is indeed stable under the passage from $\scG_n$ to $\scM_n$. We shall use the notation $\scM_n$ for the categorical operad so defined. From the decomposable object point of view, the first variable of the operad product serves as a template for how to form the iterated product of the remaining variables. For instance, writing
\[
\mu=2\Box_1(3\Box_21), \quad \nu_1=1\Box_12,\quad \nu_2=1\Box_22,\quad \nu_3=1\Box_32,
\]
the corresponding $\scM_3$ operad product is given by
\[
\gamma(\mu,\nu_1,\nu_2,\nu_3)=(3\Box_24)\Box_1((5\Box_36)\Box_2(1\Box_12)).
\]
The fact that the description of $\scM_n$ given here agrees with the original definition in \cite{BFSV} is essentially a consequence of the coherence statement in \cite[Theorem~3.6]{BFSV}.

\subsection{The restricted preoperads  $\scM_n^{\suarrow}$ and $\scM_n^{\sdarrow}$}\label{subsec:restricted-preoperads}
We define a subcategory $\scM_n^{\suarrow}(S)$ of $\scM_n(S)$ by a recursive procedure similar to that used for $\scM_n(S)$, but with the requirement that the products $\Box_i$ be applied in increasing order with respect to $i$. In detail, we again have
$\scM_n^{\suarrow}(\emptyset)=\{0\}$ and $\scM_n^{\suarrow}(\{x\})=\{x\}$. For a finite set $S$ with more than one element, assume that $\scM_n^{\suarrow}(S')$ has been defined for all proper subsets $S'\subseteq S$ and all $n\geq 1$. For each $1\leq i\leq n$, we view $\scM_i(S')$ as a subcategory of $\scM_n(S')$ in the obvious way so that also $\scM_i^{\suarrow}(S')$ becomes a subcategory of $\scM_n(S')$. We then define $\scM_n^{\suarrow}(S)$ to be the full subcategory of $\scM_n(S)$ in which an object is in the image of the composition
\[
\scM_i^{\suarrow}(S_1)\times \scM_i^{\suarrow}(S_2) \to \scM_n(S_1)\times \scM_n(S_2)\xr{\Box_i}\scM_n(S)
\]
for some ordered partition $(S_1,S_2)$ and some choice of index $1\leq i\leq n$. In a similar fashion we define a subcategory 
$\scM_n^{\sdarrow}(S)$ of $\scM_n(S)$ by requiring that the products $\Box_i$ be applied in decreasing order: The categories 
$\scM_n^{\sdarrow}(\emptyset)$ and $\scM_n^{\sdarrow}(\{x\})$ are as before and for a finite set $S$ with more than one element, we again suppose that $\scM_n^{\sdarrow}(S')$ has been defined for all proper subsets $S'$. For each $1\leq i\leq n$, we let 
$\scM_{[i,n]}(S')$ be the subcategory of $\scM_n(S')$ in which the objects have labels in $\{i,\dots,n\}$, and we write $\scM_{[i,n]}^{\sdarrow}(S')$ for the intersection $\scM_n^{\sdarrow}(S')\cap \scM_{[i,n]}(S')$. Then $\scM_n^{\sdarrow}(S)$ is the full subcategory of $\scM_n(S)$ in which an object is in the image of the composition
\[
\scM_{[i,n]}^{\sdarrow}(S_1)\times \scM_{[i,n]}^{\sdarrow}(S_2) \to \scM_n(S_1)\times \scM_n(S_2)\xr{\Box_i}\scM_n(S)
\]
for some ordered partition $(S_1,S_2)$ and index $1\leq i\leq n$. For instance, the objects
\[
3\Box_3((2\Box_16)\Box_24)\Box_3(1\Box_15) \quad \text{and}\quad 3\Box_1((2\Box_36)\Box_24)\Box_1(1\Box_35)
\]
are in $\scM^{\suarrow}_3(6)$ and $\scM^{\sdarrow}_3(6)$ respectively, whereas the object in $\scM_3(6)$ considered above is in neither of these subcategories. It follows from the definitions that there are inclusions of partially ordered sets 
\begin{equation}
\scM_n^{\suarrow}(S),\scM_n^{\sdarrow}(S)\subseteq \scM_n(S)\subseteq  \scK_n(S)\subseteq \scK_n^e(S).
\end{equation}
The operad product of $\scM_n$ does not restrict to give operads $\scM^{\suarrow}_n$ and $\scM^{\sdarrow}_n$. Instead these two families of categories each inherits the structure of a preoperad in the sense of \cite{Berger97}.
\begin{remark}
The partially ordered sets $\scM^{\suarrow}_n(S)$ and $\scM^{\sdarrow}_n(S)$ appear in various guises and with different choices of notation in the literature, eg.\ \cite{GJ94, Berger97, BFSV,Ba-Sym, AH14}. For instance, in \cite{Berger97} the preoperad $\scM^{\suarrow}_n$ is denoted $\scK(F)^{(n)}$ since it arises combinatorially from the cell structure of the classical configuration preoperad $F^{(n)}$, whereas in \cite{BFSV}, the preoperad $\scM^{\sdarrow}_n$ is denoted $\bar\scJ_n$ and called the Milgram preoperad because of its relation to the Milgram construction \cite{Milgram66}. The preoperad $\scM^{\sdarrow}_n$ is also closely related to Batanin's theory of (pruned) $n$-operads~\cite{Ba-EH}. Furthermore, the inclusions of $\scM^{\suarrow}_n$ and $\scM^{\sdarrow}_n$ in $\scM_n$ are reflected in the so-called internal $n$-operads and $n$-cooperads considered in \cite[Section~11]{Ba-EH}.

\end{remark}
\subsection{Duality properties}
We define a duality isomorphism $D\colon \scG_n(S)\to \scG_n(S)^{\op}$ by reversing the order of the labels while keeping the orientations fixed:
\[
D\mu\{x,y\}=n+1-\mu\{x,y\},\qquad D\vec\mu\{x,y\}=\vec\mu\{x,y\}.
\]
It is clear that $D$ restricts to duality isomorphisms of $\scK_n^e(S)$ and $\scK_n(S)$. Furthermore, it follows from the commutative diagrams
\[
\xymatrix@C+10pt@R-4pt{
\scK_n(S_1)\times \scK_n(S_2)\ar[d]_{D\times D} \ar[r]^-{\Box_i} & \scK_n(S)\ar[d]^{D} \\
\scK_n(S_1)^{\op}\times \scK_n(S_2)^{\op} \ar[r]^-{\Box_{n+1-i}} & \scK_n(S)^{\op}
}
\]
that there are isomorphism $D\colon\scM_n(S)\to\scM_n(S)^{\op}$ and $D\colon \scM_n^{\suarrow}(S)\to\scM_n^{\sdarrow}(S)^{\op}$. Putting all this together we get the commutative diagram
\begin{equation}\label{eq:duality-diagram}
\xymatrix@-4pt{
\scM_n^{\suarrow}(S) \ar[r] \ar[d]^{D}_{\cong}& \scM_n(S) \ar[r] \ar[d]^{D}_{\cong}& \scK_n(S) \ar[r] \ar[d]^{D}_{\cong}& 
\scK_n^e(S)\ar[d]^{D}_{\cong}\\
\scM_n^{\sdarrow}(S)^{\op} \ar[r] & \scM_n(S)^{\op} \ar[r] & \scK_n(S)^{\op} \ar[r] & \scK_n^e(S)^{\op}.
}
\end{equation}
For instance, the objects in $\scM^{\suarrow}_3(6)$ and $\scM^{\sdarrow}_3(6)$ considered in Section~\ref{subsec:restricted-preoperads} are dual in this sense.

\section{Homotopy initial and homotopy final inclusions} \label{sec:homotopy-initial-final}
Consider in general a functor $j\colon \cA\to\cB$ between small categories $\cA$ and $\cB$. We say that $j$ is \emph{homotopy final} if the geometric realization of the undercategory $(b\sdarrow j)$ is contractible for each object $b$ in $\cB$. By \cite[Theorem~19.6.7]{Hirschhorn} this implies that for any diagram $F\colon \cB\to\Top$, the canonical map 
\[
\textstyle\hocolim_{\cA}F\!\circ \!j\to\hocolim_{\cB}F
\] 
is an equivalence. (Here we use the well-known fact that homotopy colimits are homotopy invariant for all diagrams in $\Top$). 
Taking $F$ to be the terminal diagram, we recover Quillen's Theorem A: If $j\colon\cA\to\cB$ is homotopy final, then $\num{j}\colon\num{\cA}\to\num{\cB}$ is an equivalence. We shall shortly need the following lemma and for completeness we provide a proof based on Quillen's Theorem A. 

\begin{lemma}\label{lem:2-of-3-homotopy-final}
Let $\cZ$ be a partially ordered set and let $\cX\subseteq\cY\subseteq\cZ$ be full subcategories, where we assume that the inclusion $h\colon \cX\to\cZ$ is homotopy final. Then the inclusions $i\colon \cX\to\cY$ and $j\colon\cY\to\cZ$ are also homotopy final.
\end{lemma}
\begin{proof}
We first observe that for each $z\in \cZ$, the undercategory $(z\sdarrow h)$ can be identified with the full subcategory $\cX_{\geq z}$ given by the elements $x\in \cX$ such that $x\geq z$. Similarly, for each $y\in\cY$, the undercategory $(y\sdarrow i)$ can be identified with the full subcategory $\cX_{\geq y}$. Since we assume $h$ to be homotopy final, the geometric realization of $\cX_{\geq y}$ is contractible which shows that $i$ is homotopy final. In order to show that $j$ is homotopy final, we must verify that the geometric realization of $(z\sdarrow j)\cong \cY_{\geq z}$ is contractible for each $z\in \cZ$. Let $k\colon \cX_{\geq z}\to\cY_{\geq z}$ be the inclusion and note that $(y\sdarrow k)$ is isomorphic to $\cX_{\geq y}$ for each $y\in\cY_{\geq z}$. Hence $k$ is homotopy final so it follows from Quillen's Theorem A that $\num{\cX_{\geq z}}$ being contractible implies that also $\num{\cY_{\geq z}}$ is contractible.     
\end{proof}

Dually, we say that that a functor $j\colon \cA\to\cB$ is \emph{homotopy initial} if the geometric realization of the overcategory $(j\sdarrow b)$ is contractible for each object $b$ in $\cB$. In this case \cite[Theorem~19.6.7]{Hirschhorn} implies that with $F$ as above, the canonical map
\[
\textstyle\holim_{\cB}F\to \holim_{\cA}F\!\circ\! j
\]
is an equivalence. It is a formal consequence of the definitions that a functor $j\colon \cA\to\cB$ is homotopy initial if and only if $j^{\op}\colon \cA^{\op}\to \cB^{\op}$ is homotopy final. Now let $S$ be a finite set. The main point of the paper is to prove the following.

\begin{theorem}\label{thm:homotopy-final-initial}
The inclusion $\scM_n^{\suarrow}(S) \to \scK_n^e(S)$ is homotopy final and the inclusion
$\scM_n^{\sdarrow}(S) \to \scK_n^e(S)$ is homotopy initial.
\end{theorem}

Before commencing on the proof, we show how to derive Theorem~\ref{thm:intro-homotopy-initial-final} from the introduction.

\begin{proof}[Proof of Theorem~\ref{thm:intro-homotopy-initial-final}]
Suppose that $\cA$ contains $\scM_n^{\suarrow}(S)$. Using Lemma~\ref{lem:2-of-3-homotopy-final} in the case of the inclusions $\scM_n^{\suarrow}(S)\subseteq\cA\subseteq \scK_n^e(S)$, we conclude from Theorem~\ref{thm:homotopy-final-initial} that the inclusion $\cA\to \scK_n^e(S)$ is homotopy final. Another application of Lemma~\ref{lem:2-of-3-homotopy-final}, now to the inclusions $\cA\subseteq\cB\subseteq \scK_n^e(S)$, then gives the result.
The case where $\cA$ contains $\scM_n^{\sdarrow}(S)$ is similar.
\end{proof}

\begin{example}
Each inclusion in the chain
\[
\scM_n^{\suarrow}(S) \to \scM_n(S) \to \scK_n(S) \to \scK_n^e(S)
\]
is homotopy final and each inclusion in the chain
\[
\scM_n^{\sdarrow}(S) \to \scM_n(S) \to \scK_n(S) \to \scK_n^e(S)
\]
is homotopy initial.
\end{example}

\begin{remark}
It is stated in \cite{BFV07} that results from \cite{BFSV} and \cite{Berger97} imply that $\num{\scK_n}\to \num{\scK_n^e}$ is an equivalence of operads. A detailed combinatorial proof of this equivalence has been given in \cite[Theorem~4.5]{BM23}. 
\end{remark}

Applying Quillen's Theorem A to the inclusions in Theorem~\ref{thm:intro-homotopy-initial-final}, we in turn get the following. 
\begin{corollary}
Let $j\colon\cA\to\cB$ be an inclusion of full subcategories of $\scK_n^e(S)$ and suppose that $\cA$ contains one of the subcategories $\scM_n^{\suarrow}(S)$ or $\scM_n^{\sdarrow}(S)$. Then the induced map $\num{j}\colon\num{\cA}\to\num{\cB}$ is an equivalence. \qed
\end{corollary}
Finally, we state the consequence for (pre)operads which motivates our work.

\begin{theorem}\label{thm:suboperad-Kne-equivalence}
Let $\scP\to\scQ$ be an inclusion of categorical sub(pre)operads of $\scK^e_n$ and suppose that $\scP$ contains one of the preoperads $\scM_n^{\suarrow}$ or $\scM_n^{\sdarrow}$. Then $\num{\scP}\to \num{\scQ}$ is an equivalence of topological (pre)operads.\qed
\end{theorem}

We now turn to the proof of Theorem~\ref{thm:homotopy-final-initial}. Notice first that the duality diagram \eqref{eq:duality-diagram} shows it suffices to prove $\scM_n^{\sdarrow}(S) \to\scK_n^e(S)$ is homotopy initial. The technical crux in the proof is to verify the latter statement for $n=2$. In the present section we employ a formal induction argument to reduce the general case of the theorem to the case $n=2$, whereas the proof of the latter case is postponed until Section~\ref{sec:n=2proof}. 

For $n=1$, the inclusions in Theorem~\ref{thm:homotopy-final-initial} are equalities and therefore both homotopy final and homotopy initial. In the following we assume $n\geq 2$. We begin by making some general comments on the objects in 
$\scM_n^{\sdarrow}(S)$. Given a subset $S'$ of $S$ and an objects $\mu$ in $\scK_n^e(S)$, we write $\mu_{S'}$ for the restriction to an object in $\scK_n^e(S')$. If $\mu$ is an object in $\scM_n^{\sdarrow}(S)$, then $\mu_{S'}$ is an object in $\scM_n^{\sdarrow}(S')$. By an ordered partition of $S$ we understand a natural number $p\geq 1$ and a $p$-tuple $(S_1,\dots,S_p)$ of non-empty disjoint subsets of $S$ such that $\cup_{i=1}^pS_i=S$. We observe that given an object $\mu$ in 
$\scM_n^{\sdarrow}(S)$, there exists a unique ordered partition $(S_1,\dots,S_p)$ such that $\mu=\mu_{S_1}\Box_1\dots\Box_1\mu_{S_p}$ and $\mu_{S_i}$ is an object in $\scM_{[2,n]}^{\sdarrow}(S_i)$ for all $i$. Now let $\omega$ be an object in $\scK_n^e(S)$ and consider the category $(\scM_n^{\sdarrow}(S)\sdarrow \omega)$. By definition, this is the partially ordered set of objects $\mu$ in $\scM_n^{\sdarrow}(S)$ such that there exists a morphism $\mu\to\omega$ in $\scK_n^e(S)$. It follows from the above discussion that an object in $(\scM_n^{\sdarrow}(S)\sdarrow \omega)$ is uniquely determined by the following data:
\begin{itemize}
\item
an ordered partition $(S_1,\dots,S_p)$ with the property that for each pair $\{x,y\}$ in $S$ such that $\omega\{x,y\}=1$ and $\vec\omega\{x,y\}=(x,y)$, we have $x\in S_{i_1}$ and $y\in S_{i_2}$ for some $i_1<i_2$, and
 \item
for each $1\leq i\leq p$ an object $\mu_i$ in $(\scM_{[2,n]}^{\sdarrow}(S_i)\sdarrow \omega_{S_i})$.
 \end{itemize}
Notice that the first condition on $(S_1,\dots,S_p)$ implies that the restriction $\omega_{S_i}$ is indeed an object in $\scK^e_{[2,n]}(S_i)$. From this point of view there exists a morphism
\[
\big( (S_1,\mu_1),\dots,(S_p,\mu_p)\big) \to \big((T_1,\nu_1)\dots,(T_q,\nu_q)\big)
\]
in $(\scM_n^{\sdarrow}(S)\sdarrow \omega)$ if and only if
\begin{itemize}
\item
for each $1\leq i\leq p$ we have $S_i\subseteq T_j$ and  a morphism $\mu_i\to (\nu_j)_{S_i}$ in $\scM_{[2,n]}^{\sdarrow}(S_i)$ for some $1\leq j\leq q$, and  
\item
if $S_{i_1}\subseteq T_{j_1}$ and $S_{i_2}\subseteq T_{j_2}$ for $i_1<i_2$, then $j_1\leq j_2$.
\end{itemize}

Using this description, we shall exhibit $(\scM_n^{\sdarrow}(S)\sdarrow\omega )$ as a certain Grothendieck construction. We begin by reviewing the relevant definitions. Consider in general a small category $\cA$ and a functor $F\colon \cA^{\op}\to\Cat$ to the category $\Cat$ of small categories. We write $\cA\ltimes F$ for the category (the Grothendieck construction) defined as follows: An object is a pair $(a,x)$ consisting of an object $a$ in $\cA$ and an object $x$ in $F(a)$, and a morphism $(\alpha,f)\colon (a,x)\to (b,y)$ is a pair of morphisms $\alpha\colon a\to b$ in $\cA$ and $f\colon x\to F(\alpha)(y)$ in $F(a)$. Given a further morphism $(\beta,g)\colon (b,y)\to (c,z)$, composition is defined by $(\beta,g)\circ(\alpha,f)=(\beta\circ\alpha,F(\alpha)(g)\circ f)$. The identity morphism on an object $(a,x)$ is the pair of identity morphisms $(\id_a,\id_x)$. 
By Thomason's homotopy colimit theorem \cite{Thomason}, the homotopy type of the nerve $N(\cA\ltimes F)$ can be expressed in terms of the categories $F(a)$. In the present situation this takes the form of a natural equivalence
\begin{equation}\label{eq:Thomason-hocolim}
\textstyle\hocolim_{\cA^{\op}}NF\simeq N(\cA\ltimes F).
\end{equation}

\begin{remark}
The relationship between $\cA\ltimes F$ and the version of the Grothendieck construction considered by Thomason is as follows: Let $(\ )^{\op}\colon \Cat\to \Cat$ be the canonical functor that takes a small category $\cX$ to its opposite $\cX^{\op}$, and let $(F)^{\op}$ be the composite functor that takes an object $a$ in $\cA$ to $F(a)^{\op}$. The opposite category $(\cA\ltimes F)^{\op}$ can then be identified with the category $(F)^{\op}\!\int\!\cA^{\op}$ analyzed by Thomason \cite{Thomason} with the notation used in that paper. Using that the geometric realizations of $N\cX$ and $N\cX^{\op}$ are naturally homeomorphic, the equivalence in \eqref{eq:Thomason-hocolim} can thus be derived from the statement of Thomason's theorem. (Alternatively, one may rework Thomason's equivalence in the present situation).
\end{remark}

Given an object $\omega$ in $\scK^e_n(S)$, we define a partially ordered set $P(\omega)$ as follows: An object of $P(\omega)$ 
is an ordered partition $(S_1,\dots,S_p)$ with the property that for each pair $\{x,y\}$ in $S$ such that $\omega\{x,y\}=1$ and $\vec\omega\{x,y\}=(x,y)$,  we have $x\in S_{i_1}$ and $y\in S_{i_2}$ for some $i_1<i_2$. There is a morphism $(S_1,\dots,S_p)\to (T_1,\dots,T_q)$ in $P(\omega)$ if (i) for each $1\leq i\leq p$ we have $S_i\subseteq T_j$ for some $1\leq j\leq q$, and (ii) if $S_{i_1}\subseteq T_{j_1}$ and $S_{i_2}\subseteq T_{j_2}$ for $i_1<i_2$, then $j_1\leq j_2$. Using that $\omega_{S_i}$ is an object in $\scK^e_{[2,n]}(S_i)$, we define a functor $ F_{\omega}\colon P(\omega)^{\op}\to\Cat$ by setting
\[
\textstyle F_{\omega}(S_1,\dots,S_p)=\prod_{i=1}^p(\scM_{[2,n]}^{\sdarrow}(S_i)\sdarrow\omega_{S_i}).
\]
For a morphism $(S_1,\dots,S_p)\to(T_1,\dots,T_q)$ of the type described above, the functoriality of the construction is specified by requiring that for each $1\leq i\leq p$ there is a commutative diagram of the form
\[
\xymatrix@-5pt{
F_{\omega}(T_1,\dots,T_q) \ar[r] \ar[d]& F_{\omega}(S_1,\dots,S_p)\ar[d]\\
 (\scM^{\sdarrow}_{[2,n]}(T_j)\sdarrow \omega_{T_j}) \ar[r] & (\scM^{\sdarrow}_{[2,n]}(S_i)\sdarrow \omega_{S_i}). 
}
\]
Here the index $1\leq j\leq q$ is uniquely determined by the condition $S_i\subseteq T_j$, the vertical functors are the projections, and the horizontal functor at the bottom is given by restriction to $S_i$.
It follows from our description of $(\scM_n^{\sdarrow}(S)\sdarrow \omega)$ that this functor provides the promised recursive presentation of the latter as a Grothendieck construction. 

\begin{proposition}\label{prop:Jn(S)-over-Grothendieck}
There is an isomorphism
$
P(\omega)\ltimes F_{\omega} \cong (\scM_n^{\sdarrow}(S)\sdarrow \omega).
$
\qed
\end{proposition}

\begin{proof}[Proof of Theorem~\ref{thm:homotopy-final-initial} assuming the case $n=2$]
The proof is by induction on $n$, assuming the cases $n=1,2$. Using the description of  $(\scM_n^{\sdarrow}(S)\sdarrow \omega)$ as a Grothendieck construction in Proposition~\ref{prop:Jn(S)-over-Grothendieck}, Thomason's homotopy colimit theorem \eqref{eq:Thomason-hocolim} gives rise to an equivalence
\[
\textstyle N(\scM_n^{\sdarrow}(S)\sdarrow \omega)\simeq \hocolim_{P(\omega)^{\op}}NF_{\omega}.
\]
For each object $(S_1,\dots,S_p)$, we have
\[
\textstyle NF_{\omega}(S_1,\dots,S_p)=\prod_{i=1}^pN(\scM_{[2,n]}^{\sdarrow}(S_i)\sdarrow\omega_{S_i}),
\]
and arguing by induction, we may assume each of the factors on the right hand side to be contractible.
There results an equivalence
\[
\textstyle N(\scM_n^{\sdarrow}(S)\sdarrow \omega)\simeq \hocolim_{P(\omega)^{\op}}*\simeq NP(\omega).
\]
Finally, we observe that $P(\omega)$ itself has the form $(\scM_2^{\sdarrow}(S)\sdarrow \omega')$ for a suitable chosen object $\omega'$ in $\scK_2^e(S)$. Consider the binary relation on $S$ defined by $x<_1y$ if $\omega\{x,y\}=1$ and $\vec\omega\{x,y\}=(x,y)$. Since this relation is acyclic by assumption, there exists an extension to a linear ordering of $S$ for which we use the same notation. We then define $\omega'$ by
\[
\omega'\{x,y\}=
\begin{cases}
1,& \text{if $\omega\{x,y\}=1$},\\
2,& \text{if $\omega\{x,y\}\geq 2$},
\end{cases}
\]
and $\vec\omega'\{x,y\}=(x,y)$ if $x<_1y$. Notice that if $(S_1,\dots,S_p)$ is an object in $P(\omega)$, then the category $(\scM_{[2,2]}^{\sdarrow}(S_i)\sdarrow \omega'_{S_i})$ has the identity of $\omega'_{S_i}$ as its only object. This shows that $P(\omega)$ can be identified with $(\scM_2^{\sdarrow}(S)\sdarrow \omega')$ which in turn concludes the argument
\end{proof}

\section{The proof of Theorem~\ref{thm:homotopy-final-initial} for $n=2$}\label{sec:n=2proof}
Let $S$ be a finite set and let $\omega$ be a fixed object of $\scK^e_2(S)$. In this section we write $P(\omega)$ for the partially ordered set $(\scM_2^{\sdarrow}(S)\sdarrow \omega)$ and finish the proof of Theorem~\ref{thm:homotopy-final-initial} by showing that its geometric realization is  contractible. Our first objective will be to define a certain abstract simplicial complex (in the classical sense \cite[Chapter~3]{Spanier}) which represents the homotopy type of $\num{P(\omega)}$. We shall then show this simplicial complex to be contractible by exhibiting a sequence of elementary collapses which reduces it to a single element. 

Recall that an element in $P(\omega)$ is an ordered partition $\ang{S_i}_{i=1}^p$ with the property that for each pair $\{x,y\}$ in $S$ such that $\omega\{x,y\}=1$ and $\vec\omega\{x,y\}=(x,y)$, we have $x\in S_{i_1}$ and $y\in S_{i_2}$ for some $i_1<i_2$. There is a morphism $\ang{S_i}_{i=1}^p\to \ang{T_i}_{i=1}^q$ if and only if (i) for each $1\leq i\leq p$ we have $S_i\subseteq T_j$ for some $1\leq j\leq q$, and (ii) if $S_{i_1}\subseteq T_{j_1}$ and $S_{i_2}\subseteq T_{j_2}$ for $i_1<i_2$, then $j_1\leq j_2$. We introduce a binary relation $\Cap$ on the set $P(\omega)$ as follows: Given elements $v=\ang{S_i}_{i=1}^p$ and $w=\ang{T_i}_{i=1}^q$, we write $v\Cap w$ if the condition $S_{i_1}\cap T_{j_2}= \emptyset$ or $S_{i_2}\cap T_{j_1}=\emptyset$ is satisfied for each set of indices $1\leq i_1<i_2\leq p$ and $1\leq j_1< j_2\leq q$. This relation is reflexive and symmetric but not transitive in general. With $v$ and $w$ as above, consider the set
\[
I_{(v,w)}=\{(i,j)\in \{1,\dots,p\}\times \{1,\dots,q\}\colon S_i\cap T_j\neq \emptyset\}.
\] 
Notice that if the relation $v\Cap w$ holds, then the product ordering of the elements $(i,j)$ restricts to a linear ordering of $I_{(v,w)}$. Hence the tuple
\begin{equation}\label{eq:v-w-infimum}
\ang{S_i\cap T_j\colon (i,j)\in I_{(v,w)}}
\end{equation}
defines an ordered partition of $S$. It is not difficult to verify the following lemma.
\begin{lemma}
If the relation $v\Cap w$ holds, then \eqref{eq:v-w-infimum} defines an element in $P(\omega)$ which represents the infimum of $v$ and $w$.\qed
\end{lemma}

For each element $v$ in $P(\omega)$, let $P_v$ denote the overcategory $(P(\omega)\sdarrow v)$ with the induced partial ordering. Clearly these subsets define a covering of $P(\omega)$. 

\begin{lemma}\label{lem:intersection-Cap}
The intersection $P_v\cap P_w$ is non-empty if and only if the relation $v\Cap w$ holds. In this case $P_v\cap P_w=P_u$, where $u$ is the infimum \eqref{eq:v-w-infimum} of $v$ and $w$.
\end{lemma}
\begin{proof}
If the relation $v\Cap w$ holds, then the element $u$ defined by \eqref{eq:v-w-infimum} is in $P_v\cap P_w$ which is therefore non-empty. Conversely, if the condition for the relation $v\Cap w$ fails, then there is no element $u$ that fits in a diagram $v\gets u\to w$ in $P(\omega)$.
\end{proof}

Now we use the binary relation $\Cap$ to define a simplicial complex $CP(\omega)$. This has $P(\omega)$ as its set of vertices and a non-empty subset $V\subseteq P(\omega)$ is a simplex if the relation $v\Cap w$ holds for each pair of elements $v,w\in V$.  The next proposition allows us to concentrate on the symmetric relation $\Cap$ instead of the somewhat complicated notion of a morphism in $\scM_2^{\sdarrow}(S)$.

\begin{proposition}
The geometric realization of the simplicial complex $CP(\omega)$ is equivalent to $\num{P(\omega)}$.
\end{proposition}

\begin{proof}
We shall deduce this from the nerve lemma as stated by Bj\"{o}rner~\cite{Bjorner81}. Consider the covering of  $NP(\omega)$ by the nerves $NP_v$ and observe that each finite intersection
\[
NP_{v_1}\cap\dots\cap NP_{v_k}=N(P_{v_1}\cap\dots\cap P_{v_k})
\]
is either empty or contractible. Indeed, using Lemma~\ref{lem:intersection-Cap}, it follows by induction that if $P_{v_1}\cap\dots\cap P_{v_k}$ is non-empty, then it equals $P_u$ where $u$ is the infimum of the set $\{v_1,\dots,v_k\}$. By the nerve lemma it remains to show that $\{v_1,\dots,v_k\}$ is a simplex of $CP(\omega)$ if and only if $P_{v_1}\cap\dots\cap P_{v_k}$ is non-empty. This is clear for $k=1$ and follows from Lemma~\ref{lem:intersection-Cap} for $k=2$. Consider the case $k=3$ and suppose that $P_{v_1}\cap P_{v_2}\cap P_{v_3}=\emptyset$. We claim that in this case $P_{v_i}\cap P_{v_j}=\emptyset$ for some $1\leq i,j\leq 3$. Let us write $v_1=\ang{S_i}_{i=1}^p$, $v_2=\ang{T_i}_{i=1}^q$, $v_3=\ang{U_i}_{i=1}^r$, and suppose that $P_{v_1}\cap P_{v_2}\neq\emptyset$. Then $P_{v_1}\cap P_{v_2}=P_u$ with $u$ the infimum of $v_1$ and $v_2$ defined by \eqref{eq:v-w-infimum}. Since $P_u\cap P_{v_3}=\emptyset$ we must have 
\[
(S_{i_1}\cap T_{j_1})\cap U_{h_2}\neq \emptyset \text{ and } (S_{i_2}\cap T_{j_2})\cap U_{h_1}\neq \emptyset 
\]
for some $(i_1,j_1)<(i_2,j_2)$ and $h_1<h_2$. If $i_1<i_2$ this implies that $P_{v_1}\cap P_{v_3}=\emptyset$ and if $j_1<j_2$ this implies that $P_{v_2}\cap P_{v_3}=\emptyset$. Having verified the claim for $k=3$, it now follows for general $k$ that if  $P_{v_1}\cap\dots\cap P_{v_k}$ is empty, then $P_{v_i}\cap P_{v_j}$ must be empty for some $1\leq i,j\leq k$: If $P_{v_1}\cap P_{v_2}\neq\emptyset$, then $P_{v_1}\cap P_{v_2}=P_u$ as above so that $P_u\cap P_{v_3}\cap\dots\cap P_{v_k}$ is empty and we may argue by induction.
\end{proof}

Consider in general a simplicial complex $K$ with vertex set $P$. Recall that a simplex in $K$ (that is, a finite non-empty subset of $P$) is said to be a \emph{free face} if it is properly contained in exactly one other (necessarily maximal) simplex. Thus, a free face is obtained from a maximal simplex by removing a single vertex and it does not arise by removing vertices from any other maximal simplex. An \emph{elementary collapse} in $K$ is the process of removing a pair of simplices given by a free face and the unique maximal simplex containing it. It is well-known and easy to prove that an elementary collapse does not change the homotopy type represented by a simplicial complex. 

In the case of the simplicial complex $CP(\omega)$ there are usually many free faces and any elementary collapse will leave us with a smaller complex representing the same homotopy type. However, in order to exhibit a sequence of elementary collapses reducing $CP(\omega)$ to a single vertex, we will need a systematic way of recognizing free faces so that the process can be iterated. For this reason we now introduce a new partial ordering on $P(\omega)$, different from the original partial ordering inherited from $\scM_2^{\sdarrow}(S)$. Notice that an element in $P(\omega)$ can be uniquely represented by a pair $(p,\alpha)$, where $p\geq 1$ is a natural number and $\alpha\colon S\to \{1,\dots,p\}$ is a surjective function with the property that if $\omega\{x,y\}=1$ and $\vec\omega\{x,y\}=(x,y)$, then $\alpha(x)<\alpha(y)$. Indeed, given such a pair, the corresponding ordered partition $\ang{S_i}_{i=1}^p$ is defined by the inverse images $S_i=\alpha^{-1}\{i\}$. With this notation, the relation $(p,\alpha)\Cap(q,\beta)$ holds if and only if for each pair $\{x,y\}$ in $S$ such that $\alpha(x)<\alpha(y)$, we have that $\beta(x)\leq \beta(y)$. We define a new partial ordering $\peq$ on $P(\omega)$ by 
\[
(p,\alpha)\peq(q,\beta) \quad \text{if} \quad \alpha(x)\leq \beta(x) \quad \text{for all $x\in S$}. 
\]
Notice that if $(p,\alpha)\preccurlyeq(q,\beta)$, then $p\leq q$. In terms of the ordered partition notation for $v=\ang{S_i}_{i=1}^p$ and $w=\ang{T_i}_{i=1}^q$, the relation $v\peq w$ amounts to $T_i\subseteq S_1\cup \dots\cup S_i$ for $i=1,\dots,q$ (where we adopt the convention $S_i=\emptyset$ if $p<i$).

\begin{remark}
If there exists a morphism $w\to v$ in the original partial ordering of $P(\omega)$ (inherited from $\scM_2^{\sdarrow}(S)$), then $v\peq w$. The converse does not hold in general. 
\end{remark}

Next we introduce an operation on $P(\omega)$ that to a pair of elements $v$ and $w$ associates an element $v\wedge w$. Write $v=(p,\alpha)$ and $w=(q,\beta)$, and let the function $\min\{\alpha,\beta\}\colon S\to \mathbb N_{\geq 1}$ be defined by taking $x$ to $\min\{\alpha(x),\beta(x)\}$.
Notice that if $\omega\{x,y\}=1$ and $\vec\omega\{x,y\}=(x,y)$, then 
\[
\min\{\alpha(x),\beta(x)\}<\min\{\alpha(y),\beta(y)\}.
\] 
We define $v\wedge w=(r,\alpha\wedge \beta)$ to be the element in $P(\omega)$ determined by the condition that there is a commutative diagram
\[
\xymatrix@C+10pt@R-5pt{
S \ar[r]^-{\alpha\wedge \beta} \ar[dr]_{\min\{\alpha,\beta\}\ }& \{1,\dots,r\}\ar[d]^\rho\\
& \mathbb N_{\geq 1}
}
\]
with $\rho$ being injective and order preserving. Clearly the inequalities  $v\wedge w\peq v$ and $v\wedge w\peq w$ hold in general. 

\begin{example}
Let $S=\{1,2,3,4\}$ and consider the ordered partitions 
\[
v=(\{1\},\{2\},\{3,4\}),\quad w=(\{2\},\{1\},\{3,4\}).
\]
Let $\omega$ be any object of $\scK^e_2(S)$ such that $v$ and $w$ belong to $P(\omega)$. In this case we find that $v\wedge w=(\{1,2\},\{3,4\})$. Notice that this is not the infimum of $v$ and $w$ with respect to $\peq$. In fact, both elements $(\{1,2\},\{3\},\{4\})$ and $(\{1,2\},\{4\},\{3\})$ are maximal lower bounds for $v$ and $w$, so the infimum of $v$ and $w$ with respect to the relation $\peq$ does not exist.
\end{example}

\begin{lemma}
If the relation $v\Cap w$ holds, then $v\wedge w$ is the infimum of $v$ and $w$ with respect to $\peq$. 
\end{lemma}
\begin{proof}
Let $v=(p,\alpha)$ and $w=(q,\beta)$, and write $r=\min\{p,q\}$. We claim that $v\wedge w$ is the element 
$(r,\min\{\alpha,\beta\})$ which gives the result. Thus, we must show that $\min\{\alpha,\beta\}$ defines a surjective function $S\to \{1,\dots,r\}$. Given $1\leq i\leq r$, choose $x$ and $y$ such that $\alpha(x)=i$ and $\beta(y)=i$. Then the relation $v\Cap w$ ensures that $\min\{\alpha,\beta\}$ maps at least one of the elements $x$ and $y$ to $i$.
\end{proof}
In the following it will be useful to consider the operation $v\wedge w$ also in cases where it does not represent the infimum of $v$ and $w$.

\begin{lemma}\label{lem:wedge-preserves-Cap}
Given elements $u$, $v$, and $w$ in $P(\omega)$, the relations $u\Cap v$ and $u\Cap w$ imply $u\Cap(v\wedge w)$.
\end{lemma}
\begin{proof}
Write $u=(r,\gamma)$, $v=(p,\alpha)$, $w=(q,\beta)$, and assume $\gamma(x)<\gamma(y)$. Then $\alpha(x)\leq \alpha(y)$ and $\beta(x)\leq \beta(y)$, which imply
\[
\min\{\alpha(x),\beta(x)\}\leq \min\{\alpha(y),\beta(y)\}.
\]
It follows that $(\alpha\wedge\beta)(x)\leq (\alpha\wedge\beta)(y)$ as required.
\end{proof}

\begin{proposition}
\begin{description}

\item[(i)]
$P(\omega)$ has a least element with respect to $\peq$.

\item[(ii)]
Every maximal simplex in $CP(\omega)$ has a least vertex with respect to $\peq$. 
\end{description}
\end{proposition}
\begin{proof}
In (i) we may apply the operation $\wedge$ iteratively to the elements of $P(\omega)$ to produce the least element. 
For (ii) we apply the same procedure to the vertices of a maximal simplex in which case Lemma~\ref{lem:wedge-preserves-Cap} ensures that the least element remains within the maximal simplex.
\end{proof}

\begin{remark}
The statement in (i) is in contrast to the situation for the original partial ordering on $P(\omega)$ where there is usually neither a least nor a greatest element. We shall eventually exhibit a sequence of elementary collapses that reduces the simplicial complex $CP(\omega)$ to its least element with respect to $\peq$. As to the statement in (ii) we remark that for a simplex which is not maximal there may not be a least vertex.
\end{remark}

The next lemma provides the basic combinatorial input on $CP(\omega)$.

\begin{lemma}\label{lem:u-v-tilde-v-existence}
Let $u$ and $v$ be elements in $P(\omega)$ and assume that $u\speq v$. Then there exists an element $\tilde v$ in $P(\omega)$ such that (i) $\tilde v\speq v$, (ii) the relation $v\Cap \tilde v$ holds, and (iii) for an element $w$ in $P(\omega)$ the relations $u\Cap w$ and $v\Cap w$ imply $\tilde v\Cap w$.
\end{lemma}
\begin{proof}
Let us write $u=\ang{R_i}_{i=1}^k$ and $v=\ang{S_i}_{i=1}^p$, and let $2\leq j\leq p$ be such that $S_i\subseteq R_i$ for $1\leq i<j$ and $S_j\nsubseteq R_j$. We define a tuple of disjoint subsets $\tilde S_i$ of $S$ by setting
\[
\ang{\tilde S_i}_{i=1}^p=(S_1,\dots,S_{j-2},S_{j-1}\cup ((R_1\cup\dots\cup R_{j-1})\cap S_j),R_j\cap S_j,S_{j+1},\dots,S_p).
\]
Then we define
\[
\tilde v=
\begin{cases}
(\tilde S_1,\dots,\tilde S_p),& \text{if $R_j\cap S_j\neq \emptyset$,}\\
(\tilde S_1,\dots,\tilde S_{j-1},\tilde S_{j+1},\dots,\tilde S_p), & \text{if $R_j\cap S_j=\emptyset$.}
\end{cases}
\]
Let $\{x,y\}$ be elements in $S$ and suppose $\omega\{x,y\}=1$ and $\vec\omega\{x,y\}=(x,y)$. Then there are uniquely defined indices $i_1<i_2$ and $j_1<j_2$ such that $x\in R_{i_1}\cap S_{j_1}$ and  $y\in R_{i_2}\cap S_{j_2}$. We must show that $x\in \tilde S_{h_1}$ and $y\in \tilde S_{h_2}$ for some $h_1<h_2$. This is clear if $(j_1,j_2)\neq (j-1,j)$. In the remaining case $(j_1,j_2)=(j-1,j)$, we know that by assumption $S_{j-1}\subseteq R_{j-1}$ and $S_j\subseteq R_1\cup\dots\cup R_j$. Since $u$ is an element in $P(\omega)$, this implies $y\in R_j$ and hence $x\in \tilde S_{j-1}$ and $y\in \tilde S_j$. We conclude that $\tilde v$ indeed defines an element of $P(\omega)$. The relation in (i) is implied by the assumption on $j$ and the relation in (ii) is easily verified. Now suppose that $w$ is an element in $P(\omega)$ such that the relations $u\Cap w$ and $v\Cap w$ hold. Writing $w=\ang{T_i}_{i=1}^q$, we must show that $\tilde S_{i_1}\cap T_{j_2}=\emptyset$ or $\tilde S_{i_2}\cap T_{j_1}=\emptyset$ for each set of indices $i_1<i_2$ and $j_1<j_2$. This is again clear if $(i_1,i_2)\neq (j-1,j)$. In the case $(i_1,i_2)= (j-1,j)$, suppose that 
\[
\tilde S_{i_2}\cap T_{j_1}=R_j\cap S_j\cap T_{j_1}\neq \emptyset.
\]
Then $u\Cap w$ implies $R_h\cap T_{j_2}=\emptyset$ for $h<j$ and $v\Cap w$ implies $S_h\cap T_{j_2}=\emptyset$ for $h<j$. Taken together, this shows that 
\[
\tilde S_{i_1}\cap T_{j_2}=S_{j-1}\cap T_{j_2}\cup(R_1\cup\dots\cup R_{j-1})\cap S_j\cap T_{j_2}=\emptyset,
\]
which concludes the argument.
\end{proof}

The following auxiliary definition will be needed in our recursive approach to the main result. We say that a subcomplex $K$ of $CP(\omega)$ is a \emph{good subcomplex} if it is non-empty and satisfies the following conditions:
\begin{itemize}
\item
Every maximal simplex of $K$ has a least vertex with respect to $\peq$.
\item
Let $V$ be a simplex in $K$ and suppose there exists an element $\tilde v$ in $P(\omega)$ satisfying (i) $\tilde v\speq v$ for all $v$ in $V$, (ii) the relation $\tilde v\Cap v$ holds for all $v$ in $V$, and (iii) $\tilde v$ is minimal among elements in $P(\omega)$ satisfying (i) and (ii). Then $V\cup\{\tilde v\}$ is also a simplex in $K$. 
\end{itemize}
Given a good subcomplex $K$ of $CP(\omega)$, let $\Max(K)$ denote the set of maximal simplexes in $K$. We introduce a preorder on $\Max(K)$ by writing $V\peq W$ if the least vertices $v_0$ of $V$ and $w_0$ of $W$ are related by $v_0\peq w_0$. (This preorder is not antisymmetric since two maximal simplices may share a common least vertex). By definition, $V$ is maximal in this preorder if $V\peq W$ implies $W\peq V$. This means that if $v_0$ and $w_0$ again denote the least vertices, then $v_0\peq w_0$ implies $v_0=w_0$.

\begin{lemma}\label{lem:free-in-good-subcomplex}
Let $K$ be a good subcomplex of $CP(\omega)$ and let $V$ be a maximal simplex in $K$ with least vertex $v_0$. If $V$ is maximal in the preorder $\peq$ on $\Max(K)$ and $V-\{v_0\}$ is non-empty, then the latter is a free face of $V$ and the subcomplex obtained from $K$ by removing $V$ and $V-\{v_0\}$ is again a good subcomplex.
\end{lemma}
\begin{proof}
Suppose to the contrary that $W\neq V$ is a maximal simplex in $K$ and that $V-\{v_0\}\subseteq W$. Let $w_0$ be the least vertex of $W$ and let $u=v_0\wedge w_0$. Then $u\peq v_0$ and we claim that in fact $u\speq v_0$. Indeed, if $u=v_0$, then $v_0\peq w_0$, but the maximality of $V$ excludes the possibility $v_0=w_0$ and the maximality of $v_0$ excludes the possibility $v_0\speq w_0$. Notice also that $u\Cap v$ for all $v$ in $V-\{v_0\}$ by Lemma~\ref{lem:wedge-preserves-Cap}. Hence Lemma~\ref{lem:u-v-tilde-v-existence} shows the existence of an element $\tilde v_0\speq v_0$ in $P(\omega)$ satisfying the conditions (i) and (ii) in the definition of a good subcomplex (with respect to $V$). If necessary, we may implicitly replace $\tilde v_0$ by an element satisfying also the minimality condition (iii). By the definition of a good subcomplex it follows that $V\cup\{\tilde v_0\}$ is a simplex in $K$, contradicting the maximality of $V$. We conclude that there cannot be a maximal simplex $W\neq V$ containing $V-\{v_0\}$ and that consequently $V-\{v_0\}$ is a free face of $V$.

Let us write $K'$ for the subcomplex of $K$ obtained by removing $V$ and $V-\{v_0\}$. First we must show that every maximal simplex $U$ of $K'$ has a least vertex with respect to $\peq$. This is clear if $U$ is also maximal in $K$. If $U$ is not maximal in $K$, then $U\subseteq V$ and $U$ must be of the form $V-\{v\}$ for some vertex $v\neq v_0$ in $V$. Thus, in this case $v_0$ is a least vertex of $U$. Next let $U$ denote a general simplex of $K'$ and suppose that there exists an element $\tilde u$ in $P(\omega)$ such that (i) $\tilde u\speq u$ for all $u$ in $U$, (ii) the relation $\tilde u\Cap u$ holds for all $u$ in $U$, and (iii) $\tilde u$ is minimal with these properties. Then we know that $U\cup\{\tilde u\}$ is a simplex in $K$ and we must exclude the possibility that $U\cup\{\tilde u\}$ equals one of the simplices $V$ or $V-\{v_0\}$. If  $U\cup\{\tilde u\}=V$, then $\tilde u$ being the least vertex implies $\tilde u= v_0$ and hence $U=V-\{v_0\}$ which is a contradiction. If $U\cup\{\tilde u\}=V-\{v_0\}$, then $\tilde u\in V$ so that $v_0\speq \tilde u$ which contradicts the minimality condition of $\tilde u$ in (iii). We conclude that $K'$ is again a good subcomplex of $CP(\omega)$.
\end{proof}

\begin{proof}[Proof of Theorem~\ref{thm:homotopy-final-initial} for $n=2$]
Let $u_0$ be the least element of $P(\omega)$ and consider in general a good subcomplex $K$ of $CP(\omega)$. Given any vertex $v$ of $K$, we may apply Lemma~\ref{lem:u-v-tilde-v-existence} repeatedly to obtain an edge path in $K$ from $v$ to $u_0$. It follows in particular that $u_0$ is a vertex of $K$ and that if $\{v\}$ is a maximal simplex in $K$, then $v=u_0$ and $K=\{u_0\}$. We shall prove that any good subcomplex can be reduced to $\{u_0\}$ by a sequence of elementary collapses. Since $CP(\omega)$ is itself a good subcomplex, this will finish the proof of the theorem. The argument is by induction on the number of simplices in $K$. If there is only one simplex, then $K=\{u_0\}$ as observed above. If $K$ has more than one simplex, then it has a maximal simplex $V$ with minimal vertex $v_0$ such that $V-\{v_0\}$ is non-empty and $V$ is maximal in the preorder $\peq$ on $\Max(K)$. Then it follows from Lemma~\ref{lem:free-in-good-subcomplex} that $V-\{v_0\}$ is a free face of $V$ and that the subcomplex obtained by removing $V$ and $V-\{v_0\}$ is again a good subcomplex. This concludes the inductive step.
\end{proof}

\section{The relation to little $n$-cubes} \label{sec:relation-little-n-cubes}
We begin by fixing notation for the little $n$-cubes operad \cite{Boardman-Vogt,May72}. Let $I$ denote the unit interval. A positive  affine embedding $f\colon I\to I$ is a function of the form $f(t)=(1-t)a+tb$ for $0\leq a<b\leq 1$. A little $n$-cube is 
an $n$\nobreakdash-tuple $c=\ang{f_i}_{i=1}^n$ of positive affine embeddings which we identify with the corresponding embedding of $I^n$ into itself. We write $\scC_n(1)$ for the topological space of little $n$-cubes, topologized as the subspace of $I^{2n}$ defined by the endpoints of the affine embeddings. Now let $S$ be a finite set and let $\scC_n(S)$ be the space of $S$-indexed little $n$-cubes with separated interiors,
\[
\scC_n(S)=\{\ang{c_x}_{x\in S}\in \scC_n(1)^{\times S}: c_x(\mathrm{int} I^n)\cap c_y(\mathrm{int} I^n)=\emptyset \text{ for $x\neq y$}\}.
\]
We usually write $\scC_n(k)$ instead of $\scC_n(\{1,\dots,k\})$. 
In order to define the operad structure it is convenient to interpret an element $\ang{c_x}$ of $\scC_n(S)$ as a certain type of function $\ang{c_x}\colon (I^n)^{\sqcup S}\to I^n$, where the domain is the $S$-indexed disjoint union of copies of $I^n$. With this notation, the operad product 
\[
\gamma\colon \scC_n(k)\times \scC_n(S_1)\times \dots\times \scC_n(S_k) \to \scC_n(S_1\sqcup\dots \sqcup S_k)
\]
takes a tuple of elements $\ang{c_x}\in\scC_n(k)$ and $\ang{d^i_x}\in\scC_n(S_i)$ for $i=1,\dots,k$, to the element $\gamma(\ang{c_x},\ang{d^1_x},\dots,\ang{d^k_x})$ defined by the composition
\[
(I^n)^{\sqcup(S_1\sqcup\dots\sqcup S_k)}\cong(I^n)^{\sqcup S_1} \sqcup \dots\sqcup (I^n)^{\sqcup S_k}
\xr{\ang{d^1_x}\sqcup\dots\sqcup \ang{d^k_x}}(I^n)^{\sqcup k}\xr{\ang{c_x}} I^n.
\]
Restricted to the categories $\scC_n(k)$ for $k\geq 0$, this gives the little $n$-cubes operad as defined in \cite{May72}. The operad unit is the identity $\id_{I^n}$ in $\scC_n(1)$ and $\Sigma_k$ acts from the right on $\scC_n(k)$ by permuting the little $n$-cubes. 

Following \cite{Berger97} and \cite{BFSV}, we shall consider a certain family of contractible subspaces of $\scC_n(S)$. Given little $n$-cubes $c=\ang{f_i}_{i=1}^n$ and $d=\ang{g_i}_{i=1}^n$, we write $c<_id$ to mean that $f_i(1)\leq g_i(0)$. Notice that $c(\mathrm{int}I^n)\cap d(\mathrm{int}I^n)=\emptyset$ if and only if $c<_id$ or $d<_ic$ for some $1\leq i\leq n$. For each object $\mu$ in $\scG_n(S)$, we define a subspace $G(\mu)$ of $\scC_n(S)$ by
\[
G(\mu)=\textstyle\bigcap_{\{x,y\}}\{\ang{c_z}_{z\in S}\in \scC_n(1)^{\times S}\colon c_x<_ic_y \text{ if $\mu\{x,y\}=i$ and $\vec\mu\{x,y\}=(x,y)$}\},
\] 
where the intersection is over all pairs $\{x,y\}$ of elements in $S$. The following lemma is easily verified.

\begin{lemma}
The space $G(\mu)$ is a closed subspace of $\scC_n(S)$ and a convex subspace of $(I^{2n})^{\times S}$.\qed
\end{lemma}

\begin{lemma}\label{lem:G(mu)-non-empty}
For an object $\mu$ in $\scG_n(S)$, we have $G(\mu)\neq \emptyset$ if and only if $\mu$ is in the subcategory $\scK_n^e(S)$.
\end{lemma}
\begin{proof}
If $\mu$ is not in $\scK_n^e(S)$, then there exists an oriented cycle $x_1,\dots,x_k$ in $\mu$ such that the edges have constant label $i$, say. Since the relation $<_i$ is a strict partial order on $\scC_n(1)$, this shows that $G(\mu)=\emptyset$. Next suppose that $S$ is equipped with linear orderings $<_i$ for $i=1,\dots, n$. Then we claim that there exists an element $\ang{c_x}_{x\in S}$ in $\scC_n(1)^{\times S}$ such that $x<_iy$ implies $c_x<_ic_y$ for all $i$. In order to construct such an element one may divide $I$ in a number of subintervals equal to the cardinality of $S$ and then use the orderings to specify appropriate subcubes of $I^n$. Now let $\mu$ be an object in $\scK_n^e(S)$ and consider the binary relations $<_i$ on $S$ defined by $\mu\{x,y\}=i$ and $\mu\{x,y\}=(x,y)$. Since these relations are acyclic by assumption, they extend to linear orderings of $S$. Hence the above argument shows that $G(\mu)\neq \emptyset$.
\end{proof}

Let $V$ be a set of objects from $\scK_n^e(S)$ and let $U\subseteq V$ be a subset. Then there are closed inclusions
\begin{equation}\label{eq:G(mu)-inclusions}
\textstyle\bigcup_{\mu\in U}G(\mu)\subseteq \bigcup_{\mu\in V}G(\mu)\subseteq \scC_n(S).
\end{equation}
Notice that $\scC_n(S)$ is itself the union of the subspaces $G(\mu)$ as $\mu$ ranges over all the objects in $\scK_n^e(S)$. In the following we shall use the term \emph{cofibration} in the usual way for a map having the homotopy extension property with respect to all spaces. The next lemma can be deduced from the proof of \cite[Lemma~4.10]{BFV07}.

\begin{lemma}\label{lem:G-union-cofibration}
The inclusions in \eqref{eq:G(mu)-inclusions} are cofibrations. \qed
\end{lemma} 

\begin{lemma}\label{lem:Kne-infimum}
Let $\mu_1$ and $\mu_2$ be objects of $\scK_n^e(S)$ and suppose that $G(\mu_1)\cap G(\mu_2)$ is not empty. Then the infimum $\mu_0$ of $\mu_1$ and $\mu_2$ exists in $\scK_n^e(S)$ and there is an inclusion $G(\mu_1)\cap G(\mu_2)\subseteq G(\mu_0)$.
\end{lemma}
\begin{proof}
Given $\mu_1$ and $\mu_2$, we define $\mu_0$ by
$
\mu_0\{x,y\}=\min\{\mu_1\{x,y\},\mu_2\{x,y\}\}
$
and
\[
\vec\mu_0\{x,y\}=
\begin{cases}
\vec\mu_1\{x,y\},&\text{if $\mu_0\{x,y\}=\mu_1\{x,y\}$},\\
\vec\mu_2\{x,y\},&\text{if $\mu_0\{x,y\}=\mu_2\{x,y\}$}.
\end{cases}
\]
This is well-defined: If $\mu_1\{x,y\}=\mu_2\{x,y\}$, then the condition that $G(\mu_1)\cap G(\mu_2)$ be non-empty implies that 
$\vec\mu_1\{x,y\}=\vec\mu_2\{x,y\}$. Furthermore, since $G(\mu_1)\cap G(\mu_2)$ is contained in $G(\mu_0)$, Lemma~\ref{lem:G(mu)-non-empty} implies that $\mu_0$ is indeed an object in $\scK_n^e(S)$.
\end{proof}

Consider in general a finite partially ordered set $\cA$ and recall that an $\cA$-diagram $F\colon\cA\to\Top$ is said to be \emph{Reedy cofibrant} if for each object $a$ in $\cA$, the canonical map
$
\colim_{\partial(\cA\sdarrow a)}F\to F(a) 
$
is a cofibration. Here $\partial(\cA\sdarrow a)$ denotes the subcategory of the overcategory $(\cA\sdarrow a)$ obtained by excluding the terminal object $\id_a$. The property of a Reedy cofibrant diagram $F$ that will be important for us is that in this case the canonical map $\hocolim_{\cA}F\to\colim_{\cA}F$ is an equivalence, cf.\ \cite[Proposition~6.9]{BFSV}. 

Following \cite{BFV07} we define a $\scK_n^e(S)$-diagram 
\begin{equation}\label{eq:F-Kne-diagram}
F\colon \scK_n^e(S)\to \Top,\quad F(\nu)=\textstyle\bigcup_{\mu\to\nu}G(\mu),
\end{equation}
where the union is over all the morphism $\mu\to\nu$ in $\scK_n^e(S)$ and $F(\nu)$ is topologized as a subspace of $\scC_n(S)$. A morphism $\nu_1\to \nu_2$ gives rise to a subspace inclusion $F(\nu_1)\to F(\nu_2)$ which is a cofibration by Lemma~\ref{lem:G-union-cofibration}. The statement in the next proposition is contained as a special case of \cite[Lemma~4.5]{BFV07}. However, since the notation used in that paper is somewhat different from ours, we provide a direct proof of the result we need.

\begin{proposition}[\cite{BFV07}] \label{prop:FKe-Reedy-cofibrant}
The diagram $F\colon \scK_n^e(S)\to \Top$ is Reedy cofibrant and the canonical map $\colim_{\scK_n^e(S)}F\to \scC_n(S)$ is a homeomorphism.
\end{proposition}
\begin{proof}
Consider more generally a subcategory $\cA$ of $\scK_n^e(S)$ that is downward closed in the sense that if $\nu$ is in $\cA$ and $\mu\to\nu$ is a morphism in $\scK_n^e(S)$, then $\mu$ is also in $\cA$. We claim that in this situation the canonical map $\colim_{\cA}F\to \cup_{\mu\in\cA}G(\mu)$ is a homeomorphism. Indeed, using Lemma~\ref{lem:Kne-infimum} we may define an inverse as indicated by the diagram
\[
\xymatrix@-6pt{
& G(\mu_1) \ar[r] & F(\mu_1) \ar[dr]& \\
G(\mu_1)\cap G(\mu_2) \ar[r] \ar[ur] \ar[dr]& G(\mu_0) \ar[r] & F(\mu_0) \ar[u]\ar[d]\ar[r] &\colim_{\cA}F.\\
& G(\mu_2) \ar[r] & F(\mu_2) \ar[ur]
}
\]
Combined with Lemma~\ref{lem:G-union-cofibration}, this gives the statement in the proposition.
\end{proof}

Given a subcategory $\cA$ of $\scK_n^e(S)$, one may generalize the above construction by defining an $\cA$-diagram
\begin{equation}\label{eq:F-A-diagram}
F_{\cA}\colon \cA\to \Top,\quad F_{\cA}(\nu)=\textstyle\bigcup_{\mu\to\nu}G(\mu),
\end{equation}
where now the union is over all morphisms $\mu\to \nu$ in $\cA$. The case $\cA=\scM_n(S)$ is analyzed in \cite[Corollary~6.8]{BFSV} where it is stated  that the analogue of Proposition~\ref{prop:FKe-Reedy-cofibrant} holds in this case. The next example shows that this is not true in general. 

\begin{example}\label{ex:FMn-diagram-not-Reedy-cof}
We consider the category $\scM_3(3)$ and write $F_{\!\scM}\colon \scM_3(3)\to \Top$ for the corresponding diagram defined as in \eqref{eq:F-A-diagram}. It turns out that this diagram is not Reedy cofibrant. Consider for instance the objects 
\[
\mu_1=1\Box_22\Box_23, \quad \mu_2=2\Box_3(1\Box_13),\quad \nu=2\Box_3(1\Box_23)
\]
and the diagram $\mu_1\to\nu\leftarrow\mu_2$ in $\scM_3(3)$. Let $(c_1,c_2,c_3)$ be the element in $\scC_3(3)$ defined by the subcubes
\[
\textstyle c_1=[0,\frac{1}{2}]\times[0,\frac{1}{3}]\times[\frac{1}{2},1],\quad 
c_2=[0,1]\times[\frac{1}{3},\frac{2}{3}]\times[0,\frac{1}{2}],\quad
c_3=[\frac{1}{2},1]\times[\frac{2}{3},1]\times[\frac{1}{2},1],
\]
and observe that $(c_1,c_2,c_3)$ is an element in $G(\mu_1)\cap G(\mu_2)$. If $\mu_2\to\mu\to\nu$ are morphisms in 
$\scM_3(3)$, then $\mu=\mu_2$ or $\mu=\nu$. Furthermore, if $\mu\to \mu_2$ is a morphism in $\scM_3(3)$ and $(c_1,c_2,c_3)$ is in $G(\mu)$, then $\mu=\mu_2$. This shows that the left hand part of the diagram 
\[
\xymatrix@-6pt{
& F_{\!\scM}(\mu_1) \ar[d] \ar[dr] & \\
F_{\!\scM}(\mu_1)\cap F_{\!\scM}(\mu_2) \ar[ur] \ar[dr] &  \colim_{\partial(\scM_3(3)\sdarrow\nu)}F_{\!\scM} \ar[r] & F_{\!\scM}(\nu)\\
& F_{\!\scM}(\mu_2) \ar[u] \ar[ur] & 
}
\]
is not commutative. Since the outer diagram is commutative by definition, we conclude that the horizontal map on the right is not injective. Hence this map is not a cofibration. One can show by a similar argument that the $\scM_3(3)$-diagram obtained by restricting the $\scK^e_3(3)$-diagram in \eqref{eq:F-Kne-diagram} is also not Reedy cofibrant.
\end{example}

\begin{remark}
It is also claimed in \cite{BFSV} that the analogue of Proposition~\ref{prop:FKe-Reedy-cofibrant} holds for the category $\scK_n(S)$, but again this is not true in general: Modifying the argument in the above example to the situation considered in \cite[Remark~6.6]{BFV07} reveals that the corresponding diagram is not Reedy cofibrant. The reason why the argument from Proposition~\ref{prop:FKe-Reedy-cofibrant} does not carry over to $\scM_n$ and $\scK_n$ is that the analogues of Lemma~\ref{lem:Kne-infimum} fail in these cases (contrary to what is claimed in \cite[Lemma~6.7]{BFSV}). It turns out that the situation improves if one restricts further to $\scM_n^{\sdarrow}(S)$, since in this case one can prove an analogue of Lemma~\ref{lem:Kne-infimum} so that the corresponding diagram is in fact Reedy cofibrant. We shall not make use of this fact.
\end{remark}

The following result was first proved by Berger~\cite[Theorem~1.16]{Berger97} with $\scK_n$ instead of $\scK^e_n$ and the $\scK^e_n$-analogue was later stated implicitly in \cite[Lemma~4.9]{BFV07}. For completeness, we write out the $\scK^e_n$-version of Berger's proof. 

\begin{proposition}\label{prop:F(nu)-contractible}
The space $F(\nu)$ is contractible for any object $\nu$ in $\scK_n^e(S)$.
\end{proposition}
\begin{proof}
In the following we shall view $F(\nu)$ as a subspace of $\scC_n(1)^{\times S}$. For an element 
$\ang{c_x}_{x\in S}$ of the latter, given by little $n$-cubes $c_x=\ang{f^x_i}_{i=1}^n$, the condition for this to be in $F(\nu)$ is that for each pair $\{x,y\}$ such that $\vec\nu\{x,y\}=(x,y)$, we have
\[
f^x_i(1)\leq f^y_i(0) \text{ for some } i \leq\nu\{x,y\}, \quad\text{or}\quad f^y_i(1)\leq f^x_i(0) \text{ for some } i<\nu\{x,y\}.
\]
Let us fix an element $\ang{d_x}_{x\in S}$ in $G(\nu)$, given by little $n$-cubes $d_x=\ang{g^x_i}_{i=1}^n$. We define a sequence of homotopies 
$H_j\colon \scC_n(1)^{\times S}\times I\to \scC_n(1)^{\times S}$ for $1\leq j\leq n$, by
\[
H_j(\ang{f^x_1,\dots,f^x_n}_{x\in S},t)=\ang{f^x_1,\dots,f^x_{j-1},(1-t)f^x_j+tg^x_j,g^x_{j+1},\dots,g^x_n}_{x\in S}.
\]
It follows from the above description of $F(\nu)$ that these homotopies restrict to homotopies $H_j\colon F(\nu)\times I\to F(\nu)$. Furthermore, we observe that $H_n(-,0)$ is the identity function on $F(\nu)$, that $H_j(-,1)=H_{j-1}(-,0)$ for $2\leq j\leq n$, and that $H_1(-,1)$ is constant, mapping every element in $F(\nu)$ to $\ang{d_x}_{x\in S}$. Hence these homotopies combine to define a null-homotopy of $F(\nu)$.
\end{proof}

Finally, we review the argument from the proof of \cite[Lemma~4.13]{BFV07} stating that $\num{\scK_n^e}$ is an $E_n$-operad. Given a finite set $S$, consider the $\scK_n^e(S)$-diagram $F$ in \eqref{eq:F-Kne-diagram} and the equivalences
\begin{equation}\label{eq:Kne(S)-Cn(S)-equivalence}
\textstyle\num{\scK_n^e(S)}\cong\hocolim_{\scK_n^e(S)}* \xl{\simeq}\hocolim_{\scK_n^e(S)}F\xr{\simeq}\colim_{\scK_n^e(S)}F\cong \scC_n(S).
\end{equation}
Here the left hand map is induced by the projection $F\to *$ onto the terminal diagram. This is an equivalence by Proposition~\ref{prop:F(nu)-contractible} and the fact that homotopy colimits preserve object-wise equivalences. The right hand map is an equivalence since $F$ is Reedy cofibrant by Proposition~\ref{prop:FKe-Reedy-cofibrant}. In the following we shall write 
\[
\textstyle F_{h\scK_n^e(k)}=\hocolim_{\scK_n^e(k)}F, \quad k\geq 0.
\]
Using the description of the homotopy colimit from \cite{Hirschhorn}, it is not difficult to see that these spaces define an operad $F_{h\scK_n^e}$ with structure maps inherited from $\scK_n^e$ and $\scC_n$. For the $\Sigma_k$-action on $F_{h\scK_n^e(k)}$, this uses that for each object $\mu$ in $\scK_n^e(k)$ and for each element $\sigma$ in $\Sigma_k$, there is a factorization
\[
\xymatrix@-5pt{
G(\mu) \ar@{.>}[r] \ar[d]& G(\mu\sigma)\ar[d]\\
\scC_n(k) \ar[r]^-{\sigma^*} & \scC_n(k).
}
\]
The definition of the operad product uses that for each tuple of objects $\mu$ in $\scK^e_n(k)$ and $\nu_i$ in $\scK^e_n(j_i)$ for $1\leq i\leq k$, there is a factorization 
\[
\xymatrix@-5pt{
G(\mu)\times G(\nu_1)\times \dots\times G(\nu_k) \ar@{.>}[r] \ar[d]& G(\gamma(\mu,\nu_1,\dots,\nu_k))\ar[d]\\
\scC_n(k)\times \scC_n(j_1)\times \dots\times \scC_n(j_k) \ar[r] & \scC_n(j_1+\dots+j_k).
}
\]
Since the operad structure of $F_{h\scK_n^e}$ is inherited from those of $\scK_n^e$ and $\scC_n$, it follows that the weak equivalences in \eqref{eq:Kne(S)-Cn(S)-equivalence} give rise to equivalences of operads
\[
\num{\scK_n^e}\xl{\simeq} F_{h\scK_n^e} \xr{\simeq} \scC_n,
\]
showing that $\num{\scK_n^e}$ is indeed an $E_n$-operad. Using this, we can complete the proof of Theorem~\ref{thm:intro-categorical-sub(pre)operad-En} from the introduction.

\begin{proof}[Proof of Theorem~\ref{thm:intro-categorical-sub(pre)operad-En}]
Let $\scP$ be a categorical sub(pre)operad of $\scK_n^e$ that contains one of the preoperads $\scM_n^{\suarrow}$ or 
$\scM_n^{\sdarrow}$. Then Theorem~\ref{thm:suboperad-Kne-equivalence} implies that the inclusion in $\scK_n^e$ gives an equivalence $\num{\scP}\to\num{\scK_n^e}$. Combining this with the equivalences considered above, we get a chain of equivalences relating $\num{\scP}$ and $\scC_n$.
\end{proof}

The properties of the functor $F$ needed for the definition of the operad $F_{h\scK_n^e}$ have been formalized in a more general setting in \cite[Lemma~6.11]{BFSV}. Applied to the operads $\scM_n$ and $\scK_n$ and the corresponding functors \eqref{eq:F-A-diagram}, it is claimed in the proof of \cite[Theorem~3.14]{BFSV} that the arguments leading to the equivalences in \eqref{eq:Kne(S)-Cn(S)-equivalence} also work in these cases. However, as we have seen that the functor in \eqref{eq:F-A-diagram} is not in general Reedy cofibrant when $\scK^e_n$ is replaced by $\scM_n$ or  $\scK_n$, these arguments are not immediately valid. In the present paper we have avoided this problem by passing through $\scK^e_n$ and relying on Theorem~\ref{thm:suboperad-Kne-equivalence}. We end by showing that our main result for a sub(pre)operad $\scP$ of $\scK^e_n$ can be formulated without passing through $\scK^e_n$. For this it seems easiest not to work with the corresponding functor \eqref{eq:F-A-diagram}, but to instead restrict the diagram $F$ in \eqref{eq:F-Kne-diagram} to the category $\scP(k)$ via the inclusion of the latter in $\scK^e_n(k)$. Using the formalism of \cite[Lemma~6.11]{BFSV}, we then get a (pre)operad $F_{h\scP}$ defined in analogy with $F_{h\scK_n^e}$, but with $\scK^e_n$ replaced by $\scP$.
\begin{corollary}
Let $\scP$ be a categorical sub(pre)operad of $\scK^e_n$ that contains one of the preoperads $\scM_n^{\suarrow}$ or $\scM_n^{\sdarrow}$. Then there are equivalences of $E_n$-(pre)operads
\[
\num{\scP} \xl{\simeq} F_{h\scP}\xr{\simeq} \scC_n.
\]
\end{corollary} 
\begin{proof}
The maps in the diagram are defined as in \eqref{eq:Kne(S)-Cn(S)-equivalence} with $\scK^e_n$ replaced by $\scP$. Here the last map $\colim_{\scP(k)}F\to \scC_n(k)$ need not be a homeomorphism if the inclusion of $\scP(k)$ in $\scK_n^e(k)$ is not final, but that does not affect the argument. (It is a homeomorphism if $\scP(k)$ contains $\scM_n^{\suarrow}(k)$). It follows from the definitions that there is a commutative diagram of (pre)operads
\[
\xymatrix@-5pt{
\num{\scP} \ar[d]& F_{h\scP} \ar[l] \ar[r] \ar[d]& \scC_n\ar@{=}[d]\\
\num{\scK_n^e} & F_{h\scK_n^e} \ar[l] \ar[r] & \scC_n.
}
\]
Here we have already observed that the horizontal maps at the bottom of the diagram are equivalences and the left hand horisontal map at the top is again an equivalence by Proposition~\ref{prop:F(nu)-contractible}. Hence the result follows from Theorem~\ref{thm:suboperad-Kne-equivalence} which shows that $\num{\scP}\to\num{\scK_n^e}$ is an equivalence.
\end{proof}

\bibliographystyle{plain}

\begin{thebibliography}{10}
\bibitem{AH14}
D.~Ayala and R.~Hepworth.
\newblock Configuration spaces and {$\Theta_n$}.
\newblock {\em Proc. Amer. Math. Soc.}, 142(7):2243--2254, 2014.

\bibitem{BFSV}
C.~Balteanu, Z.~Fiedorowicz, R.~Schw\"{a}nzl, and R.M. Vogt.
\newblock Iterated monoidal categories.
\newblock {\em Adv. Math.}, 176(2):277--349, 2003.

\bibitem{Ba-Sym}
M.A. Batanin.
\newblock Symmetrisation of {$n$}-operads and compactification of real
  configuration spaces.
\newblock {\em Adv. Math.}, 211(2):684--725, 2007.

\bibitem{Ba-EH}
M.A. Batanin.
\newblock The {E}ckmann-{H}ilton argument and higher operads.
\newblock {\em Adv. Math.}, 217(1):334--385, 2008.

\bibitem{Berger97}
C.~Berger.
\newblock Combinatorial models for real configuration spaces and
  {$E_n$}-operads.
\newblock In {\em Operads: {P}roceedings of {R}enaissance {C}onferences
  ({H}artford, {CT}/{L}uminy, 1995)}, volume 202 of {\em Contemp. Math.}, pages
  37--52. Amer. Math. Soc., Providence, RI, 1997.

\bibitem{BM23}
A. Beuckelmann and I. Moerdijk.
\newblock A small catalogue of {$E_{n}$}-operads.
\newblock In {\em Homotopical methods in geometry and physics}, volume 841 of
  {\em Contemp. Math.}, pages 41--83. Amer. Math. Soc., [Providence], RI,
  [2026] \copyright 2026.

\bibitem{Bjorner81}
A.~Bj\"orner.
\newblock Homotopy type of posets and lattice complementation.
\newblock {\em J. Combin. Theory Ser. A}, 30(1):90--100, 1981.

\bibitem{Boardman-Vogt}
J.M. Boardman and R.M. Vogt.
\newblock {\em Homotopy invariant algebraic structures on topological spaces}.
\newblock Lecture Notes in Mathematics, Vol. 347. Springer-Verlag, Berlin-New
  York, 1973.

\bibitem{BFV07}
M.~Brun, Z.~Fiedorowicz, and R.M. Vogt.
\newblock On the multiplicative structure of topological {H}ochschild homology.
\newblock {\em Algebr. Geom. Topol.}, 7:1633--1650, 2007.

\bibitem{GJ94}
E.~Getzler and J.D.S. Jones.
\newblock Operads, homotopy algebra and iterated integrals for double loop
  spaces.
\newblock Preprint, https://arxiv.org/abs/hep-th/9403055, 1994.

\bibitem{Hirschhorn}
P.S. Hirschhorn.
\newblock {\em Model categories and their localizations}, volume~99 of {\em
  Mathematical Surveys and Monographs}.
\newblock American Mathematical Society, Providence, RI, 2003.

\bibitem{May-definitions}
J.P. May.
\newblock Definitions: Operads, algebras and modules.
\newblock Available on the authors webpage
  http://math.uchicago.edu/~may/PAPERS/handout.pdf.

\bibitem{May72}
J.P. May.
\newblock {\em The geometry of iterated loop spaces}.
\newblock Springer-Verlag, Berlin, 1972.
\newblock Lectures Notes in Mathematics, Vol. 271.

\bibitem{Milgram66}
R.J. Milgram.
\newblock Iterated loop spaces.
\newblock {\em Ann. of Math. (2)}, 84:386--403, 1966.

\bibitem{Spanier}
E.H. Spanier.
\newblock {\em Algebraic topology}.
\newblock McGraw-Hill Book Co., New York-Toronto-London, 1966.

\bibitem{Thomason}
R.W. Thomason.
\newblock Homotopy colimits in the category of small categories.
\newblock {\em Math. Proc. Cambridge Philos. Soc.}, 85(1):91--109, 1979.

\end{thebibliography}

\end{document}